\documentclass[11pt]{amsart}

\usepackage{amssymb,amscd,amsthm, verbatim,amsmath,color,fancyhdr, mathrsfs, bm}
\usepackage{graphicx}
\usepackage{turnstile}

\usepackage{tikz,tikz-cd}

\usepackage{array}
\usepackage{multirow}

\usepackage{amssymb,epsfig,psfrag}
\usepackage{amsfonts, amssymb}
\usepackage{enumerate}
\usepackage[all]{xy}
\usepackage[linesnumbered,ruled]{algorithm2e}
\usepackage{setspace}
\usepackage{float}
\usepackage{amsfonts, amssymb}
\usepackage{amssymb,epsfig,psfrag}
\usepackage{here}
\usepackage{pict2e}
\usepackage{hyperref}
\usepackage{setspace}
\usepackage{url}
\usepackage{here}
\usepackage{pict2e}
\usepackage{hyperref}
\usepackage{blindtext}
\def\multiset#1#2{\ensuremath{\left(\kern-.3em\left(\genfrac{}{}{0pt}{}{#1}{#2}\right)\kern-.3em\right)}}

\hypersetup{
    colorlinks=true,
    linkcolor=blue,
    filecolor=magenta,      
    urlcolor=cyan,
}

\usepackage[letterpaper, left=2.5cm, right=2.5cm, top=2.5cm,
bottom=2.5cm,dvips]{geometry}

\newtheorem{theorem}{Theorem}
\newtheorem{lemma}[theorem]{Lemma}
\newtheorem{corollary}[theorem]{Corollary}
\newtheorem{proposition}[theorem]{Proposition}
\theoremstyle{definition}

\theoremstyle{remark}

\title[Embeddings of quadratic  spaces over the field of $p$-adic numbers]
{Embeddings of quadratic  spaces over the field of $p$-adic numbers}

\author{Semin Yoo}

\date{\today}
\address{Department of Mathematics, University of Rochester, Rochester, NY, USA}

\email{syoo19@ur.rochester.edu}
\subjclass[2010]{15A63}

\keywords{quadratic spaces over the field of $p$-adic numbers, embedding problem}

\begin{document}

\maketitle

%\doublespacing

\begin{abstract}
Nondegenerate quadratic forms over $p$-adic fields are classified by their dimension, discriminant, and Hasse invariant. This paper uses these three invariants, elementary facts about $p$-adic fields and the theory of quadratic forms to determine which types of quadratic spaces -- including degenerate cases -- can be embedded in the Euclidean $p$-adic space $(\mathbb{Q}_{p}^{n},x_{1}^{2}+\cdots+x_{n}^{2})$, and the Lorentzian space  $(\mathbb{Q}_{p}^{n},x_{1}^{2}+\cdots+x_{n-1}^{2}+\lambda x_{n}^{2})$, where $\mathbb{Q}_{p}$ is the field of $p$-adic numbers, and $\lambda$ is a nonsquare in the finite field $\mathbb{F}_{p}$. Furthermore, the minimum dimension $n$ that admits such an embedding is determined.
\end{abstract}

\section{Introduction}

Given two sets $X$ and $Y$ with some extra structure, e.g. topology, group multiplication and so on, an \textbf{embedding} $f:X\hookrightarrow Y$ is an injective map that preserves the structure of $X$ in $Y$. For example, the Nash Embedding Theorem says that any Riemannian smooth manifold $X$ of dimension $n$ can be isometrically embedded into the Euclidean space $\mathbb{R}^{2n}$, $n>0$. For a given manifold $X$, it is interesting to find the smallest $k$ such that $X$ is embedded in $\mathbb{R}^{k}$.

\smallskip
In the same spirit as the Nash Embedding Theorem, we can decide which types of quadratic spaces can be isometrically embedded in the Euclidean $p$-adic space $(\mathbb{Q}_{p}^{n},x_{1}^{2}+\cdots+x_{n}^{2})$ and the Lorentzian $p$-adic space $(\mathbb{Q}_{p}^{n},x_{1}^{2}+\cdots+x_{n-1}^{2}+\lambda x_{n}^{2})$, where $\mathbb{Q}_{p}$ is the field of $p$-adic numbers, and $\lambda$ is a nonsquare in the finite field $\mathbb{F}_{p}$. Additionally, given a quadratic space $Y$, we determine a lower bound $n$ that allows to be $Y$ embedded in $(\mathbb{Q}_{p}^{n},x_{1}^{2}+\cdots+x_{n}^{2})$ and $(\mathbb{Q}_{p}^{n},x_{1}^{2}+\cdots+x_{n-1}^{2}+\lambda x_{n}^{2})$. Note that quadratic forms yield both nondegenerate and degenerate quadratic spaces. Over finite fields, nondegenerate quadratic forms are classified by the dimension and the discriminant $d$. Two nondegenerate quadratic spaces $(V_{1},Q_{1})$ and $(V_{2},Q_{2})$ over a finite field are isometrically isomorphic if and only if dim$V_{1}=$dim$V_{2}$ and $d(Q_{1})=d(Q_{2})$. By using these invariants together with some facts from the theory of finite fields and quadratic forms, the author in \cite{Yo} finds the embedding criteria for a given quadratic space to isometrically embedded in the Euclidean finite space $(\mathbb{F}_{q}^{n},x_{1}^{2}+\cdots+x_{n}^{2})$.

\smallskip
For quadratic spaces over $p$-adic fields, we have to consider one more invariant, called the Hasse-invariant $\epsilon$. Two nondegenerate quadratic spaces $(V_{1},Q_{1})$ and $(V_{2},Q_{2})$ over $\mathbb{Q}_{p}$ are isometrically isomorphic if and only if dim$V_{1}=\text{dim}V_{2}$, $d(Q_{1})=d(Q_{2})$ and $\epsilon(Q_{1})=\epsilon(Q_{2})$. By using these invariants together with some facts from the theory of $p$-adic fields and quadratic forms, we determine which quadratic spaces over $\mathbb{Q}_{p}$ can be embedded in the Euclidean $p$-adic space $(\mathbb{Q}_{p}^{n},x_{1}^{2}+\cdots+x_{n}^{2})$ and the Lorentzian $p$-adic space $(\mathbb{Q}_{p}^{n},x_{1}^{2}+\cdots+x_{n-1}^{2}+\lambda x_{n}^{2})$. Additionally, we find the minimum dimension $n$ that admits such an embedding.

\smallskip
We encounter difficulty in this embedding problem in the case of degenerate quadratic spaces. This difficulty addresses by the author for finite fields in \cite{Yo}. In this paper, the author addresses similar difficulties in the field of $p$-adic numbers, which require new techniques.

\smallskip
\subsection*{Acknowledgements.} 
The author would like to express gratitude to Jonathan Pakianathan for helpful discussions and encouragement for this work.

\medskip

\smallskip
\section{Preliminaries}

In this section, we provide some results of the theory of quadratic forms we will use later. For more information about the theory of quadratic forms, see expository papers such as \cite{Cl}, and \cite{Co}. Also, throughout this paper, we assume that any field we discuss is not characteristic $2$.

\smallskip
We first introduce some notation that we will use throughout the paper to express quadratic forms. Since any quadratic form can be diagonalized over any field, we denote a quadratic form of dimension $n$ by $Q=$diag$(a_{1},a_{2},\cdots,a_{n})$, where $a_{i}$ is the $i$-th element along the diagonal of the matrix representation of $Q$. If the element $a$ on the diagonal of the matrix representation repeats $k$ times, we denote it by $Q=$diag$(a^{k})$. We will use $Q$ as a shorthand notation for the corresponding quadratic space $(V,Q)$ only when there is no risk of confusing the quadratic form with the quadratic space.

\smallskip
We now state Propositions about isotropic spaces. Given an isotropic space, it is well-known that there exits an isometric embedding of the hyperbolic plane into the isotropic space. The next proposition is one direction of generalization. The proof can be found in \cite{Cl}. 

\begin{proposition}\cite{Cl}\label{hyper}
Let $(V,Q)$ be a nondegenerate quadratic space and $U \subset V$ a totally isotropic subspace with basis $u_{1},\cdots,u_{m}$. Then there exists a totally isotropic subspace $U'$, disjoint from $U$, with basis $u_{1}',\cdots.u_{m}'$ such that $B(u_{i},u_{j}')=\delta(i,j).$ In particular, $\left \langle U,U' \right \rangle \cong \bigoplus_{i=1}^{m} \mathbb{H}$.
\end{proposition}
\begin{corollary}\cite{Cl}\label{0}
Let $W$ be a maximal totally isotropic subspace of a nondegenerate quadratic space $V$. Then dim$W \leq \frac{1}{2}$dim$V$. Equality holds if and only if $V$ is hyperbolic.
\end{corollary}
We rely on the two fundamental equivalent statements in the algebraic theory of quadratic forms, called \textbf{Witt's extension theorem} and \textbf{Witt's cancellation theorem}. The proofs can be found in \cite{Cl}.

\begin{theorem}[Witt's extension theorem]\label{we}\cite{Cl}
Let $X_{1},X_{2}$ be quadratic spaces and $X_{1} \cong X_{2}$, $X_{1}=U_{1} \oplus V_{1},X_{2}=U_{2}\oplus V_{2}$, and $f:V_{1} \longrightarrow V_{2}$ an isometry. Then there is an isomtery $F:X_{1}\longrightarrow X_{2}$ such that $F|_{V_{1}}=f$ and $F(U_{1})=U_{2}$.
\end{theorem}

\begin{theorem}[Witt's cancellation theorem]\label{wc}\cite{Cl}
Let $U_{1},U_{2},V_{1},V_{2}$ be quadratic spaces where $V_{1}$ and $V_{2}$ are isometrically isomorphic. If $U_{1}\oplus V_{1} \cong U_{2} \oplus V_{2}$, then $U_{1} \cong U_{2}$.
\end{theorem}

We now want to consider the classification of quadratic forms over the field of $p$-adic numbers, and introduce some definitions regarding the classification. Let $A$ and $B$ be two matrix representations of two quadratic forms respectively. Then there is an invertible matrix such that $A=P^{t}BP$, which implies that it is not true in general that the determinant of quadratic forms is not invariant up to equivalence. Therefore, it is natural to think about the following definition: The \textbf{discriminant} $d(Q)$ of $Q$ over a field $K$ is the coset of det$M$ in $K^{\times}/K^{\times 2}$, where $M$ is a matrix associated with $Q$. For nondegenerate quadratic forms over finite fields, the dimension and the discriminant are invariants up to equivalence. 
\begin{theorem}\cite{Se}
Two nondegenerate quadratic forms over a finite fields $\mathbb{F}_{q}$ are equivalent if and only if they have the same dimension and same discriminant.
\end{theorem}
Over the field of $p$-adic numbers $\mathbb{Q}_{p}$, the classification needs one more invariant regarding $\mathbb{Q}_{p}$, called the Hasse-invariant. Let $Q=\text{diag}(a_{1},a_{2},\cdots,a_{n})$ be a nondegenerate quadratic form over $\mathbb{Q}_{p}$. Then the \textbf{Hasse-invariant} of $Q$ is defined by
\[\epsilon(Q)=\prod_{i<j}\left ( a_{i},a_{j} \right )_{p},\]
where $\left ( a,b \right )_{p}$ is the Hilbert symbol defined by
\[\left ( a,b \right )_{p}=\begin{cases}
1 & \text{ if } x^{2}a+y^{2}b=1\text{ has a solution }(x,y) \text{ in }\mathbb{Q}_{p}\times \mathbb{Q}_{p} \\ 
-1 & \text{ otherwise }
\end{cases}.\]
\begin{theorem}\cite{Se}
Two nondegenerate quadratic forms over $\mathbb{Q}_{p}$ are equivalent if and only if they have the same dimension, same discriminant and same Hasse invariant.
\end{theorem}

We can classify nondegenerate quadratic forms over the field of $p$-adic numbers by using these $3$ invariants, dimension, discriminant, and Hasse-invariant. In particular, $1$-dimensional nongenerate quadratic spaces are determined by nonzero square classes. 

\begin{lemma}\label{1dim}
The nonzero square class of $\mathbb{Q}_{p}$ is $\mathbb{Q}_{p}^{\times}/(\mathbb{Q}_{p}^{\times })^{2}=\left \{ 1,\lambda,p,\lambda p \right \}$ for some nonsquare $\lambda $ in $\mathbb{F}_{p}$.
\end{lemma}
Throughout this paper, $\lambda$ means a nonzero nonsquare in a finite field $\mathbb{F}_{p}$. Thus by Lemma \ref{1dim}, there are $4$ possible types of $1$-dimensional nondegenerate quadratic forms over the field of $p$-adic numbers:
\[ \text{diag}(1),~~~~\text{diag}(p),~~~~\text{diag}(\lambda), ~~~~ \text{diag}(\lambda p) .\]
Before we move to the classification of nondegenerate quadratic forms, we give the computations for the Hilbert symbols between all parings of elements in $\mathbb{Q}_{p}^{\times}/(\mathbb{Q}_{p}^{\times })^{2}$.

\begin{table}[H]
\begin{center}
\begin{tabular}{c||c|c|c|c}
%\hline
$\left ( \cdot,\cdot  \right )$ &  $1$ & $\lambda $ & $p$  &  $\lambda p$
\\ \hline \hline
$1$ &  $1$ & $1$  & $1$  & $1$
\\ \hline
$\lambda$  & $1$  & $1$  & $-1$  & $-1$
\\ \hline
$p$   & $1$ & $-1$  & $1$  & $-1$
\\ \hline
$\lambda p$ & $1$ & $-1$ & $-1$  & $1$
\\ 
\end{tabular}
\caption{The Hilbert symbol when $p \equiv 1$ mod $4$. \cite{Ro}}
\end{center}
\end{table}

\begin{table}[H]
\begin{center}
\begin{tabular}{c||c|c|c|c}
%\hline
$\left ( \cdot,\cdot  \right )$ &  $1$ & $\lambda $ & $p$  &  $\lambda p$
\\ \hline \hline
$1$ &  $1$ & $1$  & $1$  & $1$
\\ \hline
$\lambda$  & $1$  & $1$  & $-1$  & $-1$
\\ \hline
$p$   & $1$ & $-1$  & $-1$  & $1$
\\ \hline
$\lambda p$ & $1$ & $-1$ & $1$  & $-1$
\\ 
\end{tabular}
\caption{The Hilbert symbol when $p \equiv 3$ mod $4$. \cite{Ro}}
\end{center}
\end{table}

Now, we state the classification of nondegenerate quadratic forms over the field of $p$-adic numbers.
\begin{lemma}\cite{Ro}\label{cla1}
Let $p$ be prime. If $p \equiv 1$ mod $4$, then any nondegenerate $2$-dimensional quadratic form over $\mathbb{Q}_{p}$ is equivalent to exactly one of the following forms:
\[
    \text{diag}(1,1),\text{diag}(1,\lambda),\text{diag}(1,p),\text{diag}(1,\lambda p),\text{diag}(\lambda,p), \text{diag}(\lambda,\lambda p),\text{diag}(p,\lambda p).
\]
%\begin{itemize}
%    \item diag$(1,1)$,
%    \item diag$(1,\lambda)$,
%    \item diag$(1,p)$,
%    \item diag$(1,\lambda p)$,
%    \item diag$(\lambda, p)$,
    %\item diag$(\lambda, %\lambda p)$,
    %\item diag$(p, \lambda p %)$.
%\end{itemize}
Here, diag$(\lambda , \lambda) $ diag$(p,p)$, diag$(\lambda p,\lambda p)$ are equivalent to diag$(1,1)$. On the other hand, if $p \equiv 3$ mod $4$, any nondegenerate $2$-dimensional quadratic form over $\mathbb{Q}_{p}$ is equivalent to exactly one of the following forms:
\[
    \text{diag}(1,1),\text{diag}(p,p),\text{diag}(1,\lambda),\text{diag}(1, p),\text{diag}(1,\lambda p),\text{diag}(\lambda, p),\text{diag}(\lambda,\lambda p).
\]
%\begin{itemize}
%    \item diag$(1,1)$,
%    \item diag$(p,p)$,
%    \item diag$(1,\lambda)$,
%    \item diag$(1, p)$,
%    \item diag$(1, \lambda p)$,
%    \item diag$(\lambda, p)$,
%    \item diag$(\lambda, \lambda p )$.
%\end{itemize}
\end{lemma}
Here, diag$(\lambda,\lambda)$ is equivalent to diag$(1,1)$, diag$(\lambda p, \lambda p)$ is equivalent to diag$(p,p)$ and diag$(p, \lambda p)$ is equivalent to diag$(1,\lambda)$. We can classify $3$-dimensional nondegenerate quadratic forms over the field of $p$-adic numbers in a similar way.

\begin{lemma}\cite{Ro}\label{cla2}
Let $p$ be a prime. If $p \equiv 1$ mod $4$, then any nondegenerate $3$-dimensional quadratic form over $\mathbb{Q}_{p}$ is equivalent to exactly one of the following forms:
\[\text{diag}(1,1,1),\text{diag}(\lambda,p,\lambda p),\text{diag}(1,1,\lambda),\text{diag}(1,1,p),\text{diag}(1,1,\lambda p),\text{diag}(1,\lambda,p),\text{diag}(1,\lambda,\lambda p),\text{diag}(1,p,\lambda p).\]
%\begin{itemize}
%    \item diag$(1,1,1)$,
%    \item diag$(\lambda,p, \lambda p)$,
%    \item diag$(1,1,\lambda)$,
%    \item diag$(1,1,p)$,
%    \item diag$(1,1, \lambda p)$,
%    \item diag$(1, \lambda, p)$,
%    \item diag$(1, \lambda, \lambda p)$,
%    \item diag$(1,p,\lambda p )$.
%\end{itemize}
If $p \equiv 3$ mod $4$, then any nondegenerate $3$-dimensional quadratic form over $\mathbb{Q}_{p}$ is equivalent to exactly one of the following forms:
\[\text{diag}(1,1,1),\text{diag}(1,p,p),\text{diag}(1,1,p),\text{diag}(p,p,p),\text{diag}(1,\lambda,p),\text{diag}(1,1,\lambda p),\text{diag}(p,p,\lambda ),\text{diag}(1,p,\lambda p).\]
%\begin{itemize}
%    \item diag$(1,1,1)$,
%    \item diag$(1,p,p)$,
%    \item diag$(1,1,p)$,
%    \item diag$(p,p,p)$,
%    \item diag$(1,\lambda, p)$,
%    \item diag$(1,1,\lambda p)$,
%    \item diag$(p, p, \lambda)$,
%    \item diag$(1,\lambda, \lambda p )$.
%\end{itemize}
\end{lemma}

Lastly, we give some facts from the field of $p$-adic numbers that we will use later. 
\begin{lemma}\label{co1}
Suppose that $p \equiv 3 $ mod $4$. There are $a,b,c,d$ and $e$ in $\mathbb{Q}_{p}$ such that
\begin{itemize}
    \item $a^{2}+b^{2}=\lambda$,
    \item $a^{2}+b^{2}+c^{2}=p$,
    \item $a^{2}+b^{2}+d^{2}=\lambda p$,
    \item $a^{2}+b^{2}+e^{2}=0$.
\end{itemize}
\end{lemma}
\begin{proof}
We may use $-1$ instead of $\lambda$ when $p \equiv 3$ mod $4$. For a given nonsquare $\lambda$, there exist $a,b$ in $\mathbb{F}_{p}$ such that $a^{2}+b^{2}=-1$. Let us define $f(x)=a^{2}+b^{2}+x^{2}-p$. Then we have $f(1)=0$ mod $p$ and $f'(1)=2\not\equiv 0$ mod $p$. By Hensel's Lemma, there is $c$ in $\mathbb{Z}_{p}$ such that $a^{2}+b^{2}+c^{2}=p$. Let us define $g(x)=a^{2}+b^{2}+x^{2}+p$ and $h(x)=a^{2}+b^{2}+x^{2}$. By Hensel's Lemma again, we can complete the proof in a similar way.
\end{proof}
The proof of the following Lemma is similar to the proof of Lemma \ref{co1}. We leave it for our reader as an exercise. 
\begin{lemma}\label{co2}
Let $p \equiv 1$ mod $4$. Then there are $A,B,C,D,E,F,G$ and $H$ in $\mathbb{Q}_{p}$ such that
\begin{itemize}
    \item $A^{2}+B^{2}=\lambda$,
    \item $C^{2}+D^{2}=p$,
    \item $E^{2}+F^{2}=\lambda p$,
    \item $G^{2}+H^{2}=0$.
\end{itemize}
\end{lemma}

\section{Embeddings of quadratic  spaces over the field of $p$-adic numbers}
In this section, we prove which types of quadratic spaces over the field of $p$-adic numbers can be embedded in the Euclidean $p$-adic space $(\mathbb{Q}_{p}^{n},\text{diag}(1^{n}))$, and the Lorentzian $p$-adic space $(\mathbb{Q}_{p}^{n},\text{diag}(1^{n-1},\lambda ))$. Throughout this section, I will refer to $a,b,c,d,e,A,B,C,D,E,F$ and $H$, which are all elements in the field of $p$-adic numbers and are are the same as the elements referred to in Lemma \ref{co1} and Lemma \ref{co2}.

\subsection{Embeddings in $(\mathbb{Q}_{p}^{n},\text{diag}(1^{n} ))$ when $p\equiv 1$ mod $4$}
We assume that $p \equiv 1$ mod $4$. We first consider the case of $1$-dimensional nondegenerate quadratic spaces. Suppose that $n=2$. Then all $4$ types described in Lemma \ref{1dim} always exist in $(\mathbb{Q}_{p}^{2},\text{diag}(1,1))$:
\begin{itemize}
    \item $\text{diag}(1)=\left \langle (1,0) \right \rangle$,
    \item $\text{diag}(p)=\left \langle (C,D) \right \rangle$,
    \item $\text{diag}(\lambda)=\left \langle (A,B) \right \rangle$,
    \item $\text{diag}(\lambda p)=\left \langle (E,F) \right \rangle$.
\end{itemize}
If $n \geq 3$, we only need to add more $0$'s, so these $4$ types of $1$-dimensional nondegenerate quadratic spaces always exist in $(\mathbb{Q}_{p}^{n},\text{diag}(1^{n}))$ when $n \geq 3$.

\smallskip
Let us think about $2$-dimensional nondegenerate quadratic spaces over $\mathbb{Q}_{p}$. Suppose that $n=3$. Since there are $4$ types of $1$-dimensional nondegenerate quadratic subspaces of $(\mathbb{Q}_{p}^{3},\text{diag}(1^{3}))$, there should be $4$ types of $2$-dimensional quadratic subspaces of $(\mathbb{Q}_{p}^{3},\text{diag}(1^{3}))$. By looking at Theorem \ref{cla1}, we find $4$ of $7$ types from Theorem \ref{co1}:
\begin{itemize}
    \item diag$(1,1)=\left \langle (1,0,0),(0,1,0) \right \rangle$,
    \item diag$(1,\lambda)=\left \langle (1,0,0),(0,A,B) \right \rangle$,
    \item diag$(1,p)=\left \langle (1,0,0),(0,C,D) \right \rangle$,
    \item diag$(1,\lambda p)=\left \langle (1,0,0),(0,E,F) \right \rangle$.
\end{itemize}
It is not hard to show that these $4$ types exist when $n \geq 4$. The proof is similar to above, where we insert $0$ $(n-2)$ times into each vector. If $n=4$, then the $3$ other types in Theorem \ref{cla1} also exist.
\begin{itemize}
    \item diag$(\lambda,p)=\left \langle (A,B,0,0),(0,0,C,D) \right \rangle$,
    \item diag$(\lambda,\lambda p)=\left \langle (A,B,0,0),(0,0,E,F) \right \rangle$,
    \item diag$(p,\lambda p)=\left \langle (C,D,0,0),(0,0,E,F) \right \rangle$.
\end{itemize}
Furthermore, if $n\geq 7$, all $7$ types in Theorem \ref{cla1} exist and the proof is similar to above. 

\smallskip
Now we consider $3$-dimensional nondegenerate quadratic spaces. Suppose that $n=4$. Since there are $4$ types of $1$-dimensional nondegenerate quadratic subspaces of $(\mathbb{Q}_{p}^{4},\text{diag}(1^{4}))$, there have to be $4$ types of $3$-dimensional nondegenerate quadratic subspaces of $(\mathbb{Q}_{p}^{4},\text{diag}(1^{4}))$. We find $4$ of the $7$ types from Theorem \ref{cla2}:
\begin{itemize}
    \item diag$(1,1,1)=\left \langle (1,0,0,0),(0,1,0,0),(0,0,1,0) \right \rangle$,
    \item diag$(1,1,\lambda)=\left \langle (1,0,0,0),(0,1,0,0),(0,0,A,B) \right \rangle$,
    \item diag$(1,1,p)=\left \langle (1,0,0,0),(0,1,0,0),(0,0,C,D) \right \rangle$,
    \item diag$(1,1,\lambda p)=\left \langle (1,0,0,0),(0,1,0,0),(0,0,E,F) \right \rangle$.
\end{itemize}
Similar as before, it is not hard to show that these $4$ types exist when $n \geq 5$. Suppose that $n=5$. Since there are $7$ types of $2$-dimensional nondegenerate quadratic subspaces of $(\mathbb{Q}_{p}^{5},\text{diag}(1^{5}))$, there have to be $7$ types of $3$-dimensional quadratic subspaces of $(\mathbb{Q}_{p}^{5},\text{diag}(1^{5}))$. We explicitly derive $3$ of the other types in Theorem \ref{cla2}:
\begin{itemize}
    \item diag$(1,\lambda,p)=\left \langle (1,0,0,0,0),(0,A,B,0,0),(0,0,0,C,D) \right \rangle$,
    \item diag$(1,\lambda,\lambda p)=\left \langle (1,0,0,0,0), (0,A,B,0,0),(0,0,0,E,F) \right \rangle$,
    \item diag$(1,p,\lambda p)=\left \langle (1,0,0,0,0),(0,C,D,0,0),(0,0,0,E,F) \right \rangle$.
\end{itemize}
Furthermore, if $n = 6$, all $8$ cases in Theorem \ref{cla2} exist containing
\[\text{diag}(\lambda,p,\lambda p)=\left \langle (A,B,0,0,0,0),(0,0,C,D,0,0),(0,0,0,0,E,F) \right \rangle.\] 
Similarly, all $8$ cases can be embedded in $(\mathbb{Q}^{n},\text{diag}(1^{n}))$ when $n \geq 6$. 

\smallskip
For nondegenerate quadratic spaces of dimension $n$ greater than $3$, they are determined by $3$-dimensional nondegenerate quadratic spaces. We will not mention this Theorem in other sessions again.
\begin{theorem}\cite{Ro}
Let $n \geq 3$ and $Q$ be a nondegenerate quadratic form over $\mathbb{Q}_{p}$. Then $Q$ is equivalent to diag$(1^{n-3})\oplus Q'$, where $Q'$ is some $3$-dimensional quadratic form over $\mathbb{Q}_{p}$. 
\end{theorem}

Let us think about the case of degenerate quadratic spaces. For example, it is hard to check if diag$(p,0,0)$ can be embedded in $(\mathbb{Q}_{p}^{5},\text{diag}(1^{5}))$ by finding an orthogonal basis. In addition, it is not possible to use its discriminant and Hasse-invariant since diag$(p,0,0)$ is degenerate. It turns out that the nondegenerate part of degenerate spaces determines the conditions of embedding of degenerate quadratic spaces.
\begin{theorem}\label{main}
Suppose that $p\equiv 1$ mod $4$. Let $A$ be a degenerate quadratic space isometrically isomorphic to $\text{diag}(0^{k}) \oplus S$, where $S$ is a nondegenerate quadratic space. Then $A$ is embedded in diag$(1^{n})$ if and only if $S$ is embedded in $\text{diag}(1^{n-2k})$.
\end{theorem}

\begin{proof}
Suppose that $A$ is embedded in diag$(1^{n})$ and $n$ is odd. We first note that $\mathbb{H}$ is isometrically isomorphic to diag$(1,1)$ when $p \equiv 1$ mod $4$. Thus we have $\mathbb{Q}_{p}^{n}=\bigoplus_{i=1}^{\frac{n-1}{2}}\mathbb{H}\oplus \text{diag}(1)$. 

\smallskip
Let us choose coordinates $x_{i},y_{i}$ and $z_{n}$ such that $B(x_{i},y_{j})=\delta_{i,j}$, $B(x_{i},x_{j})=B(y_{i},y_{j})=0$, $B(z_{n},z_{n})=1$, $B(x_{i},z_{n})=B(y_{j},z_{n})=0$ for any $i,j$. Then there exists an isotropic subspace $W$ such that $W \simeq \text{diag}(0^{k})\simeq \left \langle x_{1},\cdots,x_{k} \right \rangle$ and $A=W\oplus S$. By Witt's Theorem, there is an isometry $g$ in $O(n)$ such that $g(W)=\left \langle x_{1},\cdots,x_{k} \right \rangle$. Thus we have 
\[g(S) \cap g(W)=0 \text{ and } g(S) \subset \left \langle x_{1},\cdots,x_{k} \right \rangle^{\perp}=\left \langle x_{1},\cdots,x_{k},\cdots,x_{\frac{n-1}{2}},y_{k+1},\cdots,y_{\frac{n-1}{2}},z_{n} \right \rangle.\]
Let us define an orthogonal projection $\pi$ and an inclusion map $i$ as follows:
\begin{align*}
  \pi&:\left \langle x_{1},\cdots,x_{\frac{n-1}{2}},y_{k+1},\cdots,y_{\frac{n-1}{2}},z_{n} \right \rangle \longrightarrow \left \langle x_{k+1},y_{k+1},\cdots,x_{\frac{n-1}{2}},y_{\frac{n-1}{2}},z_{n} \right \rangle\\
  i&:g(S) \longrightarrow \left \langle x_{1},\cdots,x_{\frac{n-1}{2}},y_{k+1},\cdots,y_{\frac{n-1}{2}},z_{n} \right \rangle. 
\end{align*}
Since $g(S)\cap \left \langle x_{1},\cdots,x_{k} \right \rangle=0$, it follows that ker$(\pi)=\left \langle x_{1},\cdots,x_{k} \right \rangle$. This induces an injective embedding $\pi \circ i$ by the following diagram:
\begin{center}
\begin{tikzcd}
g(S) \arrow[rd,hook,"\pi \circ i"'] \arrow[r,hook, "i"] & \arrow[d,"\pi"] \left \langle x_{1},\cdots,x_{\frac{n-1}{2}},y_{k+1},\cdots,y_{\frac{n-1}{2}},z_{n} \right \rangle \\
& \left \langle x_{k+1},y_{k+1},\cdots,x_{\frac{n-1}{2}},y_{\frac{n-1}{2}},z_{n} \right \rangle
\end{tikzcd}
\end{center}
Since $\text{diag}(1^{n-2k})=\left \langle x_{k+1},y_{k+1},\cdots,x_{\frac{n-1}{2}},y_{\frac{n-1}{2}},z_{n} \right \rangle$, it follows that $S$ is embedded in diag$(1^{n-2k})$.

\smallskip
Conversely, suppose that suppose that $S$ is embedded in $ \text{diag}(1^{n-2k})$. Note that \[\text{diag}(1^{n})=\text{diag}(1^{n-2k})\oplus \bigoplus_{i=1}^{k}\mathbb{H} \supseteq \text{diag}(1^{n-2k})\oplus \left \langle x_{1},\cdots,x_{k} \right \rangle,\] 
where $\mathbb{H}=\left \langle x_{i},y_{i} \right \rangle$, $B(x_{i},y_{i})=\delta_{i,j}$ and $B(x_{i},x_{j})=B(y_{i},y_{j})=0$. Since $W$ is isometrically isomorphic to $\left \langle x_{1},\cdots,x_{k} \right \rangle$, this completes the proof. 

\smallskip
When $n$ is even, we omit the proof because it is similar to the case in which $n$ is odd.
\end{proof}

Suppose that a quadratic subspace $W$ of $(\mathbb{Q}_{p}^{n},\text{diag}(1^{n}))$ is isometrically isomorphic to diag$(0^{k})$. Then by Theorem \ref{main}, we derive the same result as in Corollary \ref{0}: dim$W\leq n/2$. In particular, if $n$ is even, a totally isotropic subspace $W$ with the dimension $n/2$ in $(\mathbb{Q}_{p}^{n},\text{diag}(1^{n}))$ can be constructed by 
\[\left \langle (G,H,0,\cdots,0),(0,0,G,H,0,\cdots,0), \cdots,(0,0,\cdots,0,0,G,H) \right \rangle.\]

\begin{corollary}
Suppose that $p \equiv 1$ mod $4$. For each $n$ and $k$, 
\begin{itemize}
    \item diag$(1)\oplus$diag$(0^{k})$ can be embedded in diag$(1^{n})$ if and only if $k \leq \frac{n-1}{2}$,
    \item diag$(\lambda)\oplus$diag$(0^{k})$ can be embedded in diag$(1^{n})$ if and only if $k \leq \frac{n-2}{2}$,
    \item diag$(p)\oplus$diag$(0^{k})$ can be embedded in diag$(1^{n})$ if and only if $k \leq \frac{n-2}{2}$,
    \item diag$(\lambda p)\oplus$diag$(0^{k})$ can be embedded in diag$(1^{n})$ if and only if $k \leq \frac{n-2}{2}$,
    \item diag$(1,1)\oplus$diag$(0^{k})$ can be embedded in diag$(1^{n})$ if and only if $k \leq \frac{n-2}{2}$,
    \item diag$(1,\lambda )\oplus$diag$(0^{k})$ can be embedded in diag$(1^{n})$ if and only if $k \leq \frac{n-3}{2}$,
    \item diag$(1,p)\oplus$diag$(0^{k})$ can be embedded in diag$(1^{n})$ if and only if $k \leq \frac{n-3}{2}$,
    \item diag$(1,\lambda p)\oplus$diag$(0^{k})$ can be embedded in diag$(1^{n})$ if and only if $k \leq \frac{n-3}{2}$,
    \item diag$(\lambda,p)\oplus$diag$(0^{k})$ can be embedded in diag$(1^{n})$ if and only if $k \leq \frac{n-4}{2}$,
    \item diag$(\lambda,\lambda p)\oplus$diag$(0^{k})$ can be embedded in diag$(1^{n})$ if and only if $k \leq \frac{n-4}{2}$,
    \item diag$(p,\lambda p)\oplus$diag$(0^{k})$ can be embedded in diag$(1^{n})$ if and only if $k \leq \frac{n-4}{2}$,
    \item diag$(1,1,1)\oplus$diag$(0^{k})$ can be embedded in diag$(1^{n})$ if and only if $k \leq \frac{n-3}{2}$, 
    \item diag$(1,1,\lambda )\oplus$diag$(0^{k})$ can be embedded in diag$(1^{n})$ if and only if $k \leq \frac{n-4}{2}$,
    \item diag$(1,1,p)\oplus$diag$(0^{k})$ can be embedded in diag$(1^{n})$ if and only if $k \leq \frac{n-4}{2}$,
    \item diag$(1,1,\lambda p)\oplus$diag$(0^{k})$ can be embedded in diag$(1^{n})$ if and only if $k \leq \frac{n-4}{2}$,
    \item diag$(1,\lambda ,p)\oplus$diag$(0^{k})$ can be embedded in diag$(1^{n})$ if and only if $k \leq \frac{n-5}{2}$,
    \item diag$(1,\lambda ,\lambda p)\oplus$diag$(0^{k})$ can be embedded in diag$(1^{n})$ if and only if $k \leq \frac{n-5}{2}$,
    \item diag$(1,p,\lambda p)\oplus$diag$(0^{k})$ can be embedded in diag$(1^{n})$ if and only if $k \leq \frac{n-5}{2}$,
    \item diag$(\lambda ,p,\lambda p)\oplus$diag$(0^{k})$ can be embedded in diag$(1^{n})$ if and only if $k \leq \frac{n-6}{2}$.
\end{itemize}
\end{corollary}

Notice that any nondegenerate quadratic space of dimension $k$ is isometrically isomorphic to diag$(1^{k-3})\oplus Q$, where $Q$ is a $3$-dimensional nondegenerate quadratic space. At this point, we have all described embedding conditions for nondegenerate and degenerate quadratic spaces. Therefore, given any quadratic space $V$, we can determine if $V$ can be embedded in $(\mathbb{Q}_{p}^{n},\text{diag}(1^{n} ))$ when $p \equiv 1$ mod $4$.

\subsection{Embeddings in $(\mathbb{Q}_{p}^{n},\text{diag}(1^{n-1},\lambda ))$ when $p\equiv 1$ mod $4$}
We first consider the case of nondegenerate quadratic spaces. Let us consider $1$-dimensional nondegenerate quadratic spaces of $(\mathbb{Q}_{p}^{n},\text{diag}(1^{n-1},\lambda ))$. Suppose that $n=2$. Then there are $2$ types of $1$-dimensional nondegenerate quadratic spaces of $(\mathbb{Q}_{p}^{n},\text{diag}(1^{n-1},\lambda ))$: $\text{diag}(1)=\left \langle (1,0) \right \rangle$ and $\text{diag}(\lambda)=\left \langle (0,1) \right \rangle$. Unlike the case in which the embedded space is $(\mathbb{Q}_{p}^{2},\text{diag}(1,1))$, the spaces diag$(p)$ and diag$(\lambda p)$ do not exist in $(\mathbb{Q}_{p}^{2},\text{diag}(1,\lambda))$.

\begin{lemma}\label{p}
If $p \equiv 1$ mod $4$, diag$(p)$ does not exist in $(\mathbb{Q}_{p}^{2},\text{diag}(1,\lambda ))$.
\end{lemma}

\begin{proof}
Suppose that there are $z,w$ in $\mathbb{Q}_{p}$ such that $z^{2}+\lambda w^{2}=p$. Let us write $z=p^{\alpha}u$ and $w=p^{\beta}v$, where $u,v$ are in $Z_{p}^{\times}$. If $\alpha<\beta$, then we have
\[p=z^{2}+\lambda w^{2}=p^{2\alpha}u^{2}+\lambda p^{2\beta}v^{2}=p^{2\alpha}(u^{2}+\lambda p^{2(\beta -\alpha)}v^{2}).\]
Since $u^{2}+\lambda p^{2(\beta-\alpha)}v^{2}$ is in $\mathbb{Z}_{p}$, this contradicts the fact that $\nu_{p}(p)=1$ and $\nu_{p}(p^{2\alpha})=2\alpha$, where $\nu_{p}$ is the $p$-adic norm. If $\alpha=\beta=-k$, where $k \geq 1$, then we obtain $p=p^{2\alpha}(u^{2}+\lambda v^{2})$, which implies $u^{2}+\lambda v^{2}=p^{1+2k}$. Thus we have $u^{2}+\lambda v^{2} \equiv 0$ mod $p$, which contradicts $- \lambda$ is a nonsquare.
\end{proof}

\begin{lemma}\label{alot}
If $p \equiv 1$ mod $4$, diag$(\lambda p)$ does not exist in $(\mathbb{Q}_{p}^{2},\text{diag}(1,\lambda ))$.
\end{lemma}

\begin{proof}
Suppose that diag$(\lambda p)$ exists. Then we find diag$(\lambda p)\oplus \text{diag}(p)=\text{diag}(1,\lambda)$ by comparing the discriminants and  the Hasse-invariants of the left and the right hand sides. On the other hand, we also have diag$(\lambda p)\oplus (\text{diag}(\lambda p))^{\perp}=\text{diag}(1,\lambda)$. By Witt's Theorem, we obtain diag$( p)=(\text{diag}(\lambda p))^{\perp}$, which contradicts the fact that diag$(p)$ does not exist in $(\mathbb{Q}_{p}^{2},\text{diag}(1,\lambda))$ by Lemma \ref{p}.
\end{proof}

If $n=3$, then diag$(p)=\left \langle (C,D,0) \right \rangle$ and diag$(\lambda p)=\left \langle (E,F,0) \right \rangle$ also exist. Similarly, when $n \geq 4$, as shown in section $1$, all $4$ types of $1$-dimensional nondegenerate quadratic spaces can be embedded in $\mathbb{Q}_{p}^{n}$. 

\smallskip
Let us think about $2$-dimensional quadratic spaces. Suppose that $n=3$. Since there are $4$ types of $1$-dimensional nondegenerate quadratic subspaces of $(\mathbb{Q}_{p}^{3},\text{diag}(1,1,\lambda))$, there should be $4$ types of $2$-dimensional nondegenerate quadratic subspaces of $(\mathbb{Q}_{p}^{3},\text{diag}(1,1,\lambda))$. Among the $7$ types from Theorem \ref{cla1}, only $4$ are determined by finding their orthogonal complements. Thus we find the following $4$ types:
\begin{itemize}
    \item diag$(1,1)=\left \langle (1,0,0),(0,1,0) \right \rangle$,
    \item diag$(1,\lambda)=( \text{diag}(1) )^{\perp}$,
    \item diag$(\lambda,p)=( \text{diag}(p) )^{\perp}$,
    \item diag$(\lambda,\lambda p)=( \text{diag}(\lambda p) )^{\perp}$.
\end{itemize}
If $n=4$, then the following $2$ other cases also exist as follows:
\begin{itemize}
    \item diag$(1,p)=( \text{diag}(\lambda,p) )^{\perp}$,
    \item diag$(1,\lambda p)=( \text{diag}(\lambda,\lambda p) )^{\perp}$.
\end{itemize}
The space diag$(p,\lambda p)$ does not exist in $(\mathbb{Q}_{p}^{4},\text{diag}(1^{3},\lambda))$. The proof is similar to Lemma \ref{alot}. If $n=5$, diag$(p,\lambda p)=\left \langle (C,D,0,0,0),(0,0,E,F,0) \right \rangle$ exists. Thus, if $n \geq 5$, all $7$ types of $2$-dimensional nondegenerate quadratic spaces can be embedded in $(\mathbb{Q}_{p}^{n},\text{diag}(1^{n-1},\lambda))$.

\smallskip
Let us think about $3$-dimensional nondegenerate quadratic spaces. Suppose that $n=4$. Since there are $4$ types of $1$-dimensional nondegenerate quadratic subspaces of $(\mathbb{Q}_{p}^{4},\text{diag}(1^{3},\lambda))$, there have to be $4$ types of $3$-dimensional nondegenerate quadratic subspaces of $(\mathbb{Q}_{p}^{4},\text{diag}(1^{3},\lambda))$. Of the $8$ types from Theorem \ref{cla2}, we find the following $4$: 
\begin{itemize}
    \item diag$(1,1,1)=\left \langle (1,0,0,0),(0,1,0,0),(0,0,1,0) \right \rangle$,
    \item diag$(1,1,\lambda)=( \text{diag}(1) )^{\perp}$,
    \item diag$(1,\lambda,p)=( \text{diag}(p) )^{\perp}$,
    \item diag$(1,\lambda,\lambda p)=( \text{diag}(\lambda p) )^{\perp}$.
\end{itemize}
If $n=5$, since there are $7$ types of $2$-dimensional nondegenerate quadratic spaces, there must be $7$ types of $3$-dimensional quadratic subspaces in $(\mathbb{Q}_{p}^{5},\text{diag}(1^{4},\lambda))$. Since we already know that the $4$ types exist if $n=4$, we only need to find $3$ other types. Similarly to above, it is not hard to find them:
\begin{itemize}
    \item diag$(1,1,1)=( \text{diag}(1,\lambda) )^{\perp}$,
    \item diag$(1,1,\lambda)=( \text{diag}(1,1) )^{\perp}$,
    \item diag$(1,\lambda,p)=( \text{diag}(1,p) )^{\perp}$,
    \item diag$(1,\lambda,\lambda p)=( \text{diag}(p,\lambda p) )^{\perp}$,
    \item diag$(\lambda,p,\lambda p)=( \text{diag}(p,\lambda p) )^{\perp}$,
    \item diag$(1,1,p)=( \text{diag}(\lambda,p) )^{\perp}$,
    \item diag$(1,1,\lambda p)=( \text{diag}(\lambda,\lambda p) )^{\perp}$.
\end{itemize}
If $n=6$, then diag$(1,p,\lambda p)=( \text{diag}(1,\lambda,\lambda p) )^{\perp}=\left \langle (1,0,0,0,0,0),(0,C,D,0,0,0),(0,0,0,E,F,0) \right \rangle$ also exists. Thus all $8$ types of $3$-dimensional nondegenerate quadratic spaces can be embedded in $(\mathbb{Q}_{p}^{n},\text{diag}(1^{n-1},\lambda))$ when $n \geq 6$.

\smallskip
We will now think about degenerate quadratic spaces. We prove the following theorem using a similar method to the one used in Theorem \ref{main}.
\begin{theorem}\label{main2}
Suppose that $p\equiv 1$ mod $4$. Let $A$ be a degenerate quadratic space isometrically isomorphic to $\text{diag}(0^{k}) \oplus S$, where $S$ is a nondegenerate quadratic space. Then $A$ is embedded in diag$(1^{n-1},\lambda)$ if and only if $S$ is embedded in $\text{diag}(1^{n-2k-1},\lambda )$ (and $k \leq (n-2)/2$ when $n$ is even). 
\end{theorem}

\begin{proof}
Suppose that $A$ is embedded in diag$(1^{n-1},\lambda)$ and $n$ is even. We first note that $\mathbb{H}$ is isometrically isomorphic to diag$(1,1)$ when $p \equiv 1$ mod $4$. Thus we have $\text{diag}(1^{n-1},\lambda)=\bigoplus_{i=1}^{\frac{n-2}{2}}\mathbb{H}\oplus \text{diag}(1,\lambda )$. 

\smallskip
Let us choose coordinates $x_{i},y_{i}$ such that $B(x_{i},y_{j})=\delta_{i,j}$, $B(x_{i},x_{j})=B(y_{i},y_{j})=0$ for any $i,j$. Then there exists an isotropic subspace $W$ such that $W \simeq \text{diag}(0^{k})\simeq \left \langle x_{1},\cdots,x_{k} \right \rangle$ and $A=W\oplus S$. By Witt's Theorem, there is an isometry $g$ in $O(n,q)$ such that $g(W)=\left \langle x_{1},\cdots,x_{k} \right \rangle$. Thus we have 
\[g(S) \cap g(W)=0 \text{ and } g(S) \subset \left \langle x_{1},\cdots,x_{k} \right \rangle^{\perp}=\left \langle x_{1},\cdots,x_{k},\cdots,x_{\frac{n-2}{2}},y_{k+1},\cdots,y_{\frac{n-2}{2}} \right \rangle.\]
Additionally, since diag$(1,\lambda)$ does not contain isotropic vectors, $g(W)$ should be contained in $\bigoplus_{i=1}^{(n-2)/2}\mathbb{H}$, which implies $k \leq (n-2)/2$. 

\smallskip
Let us define an orthogonal projection $\pi$ and an inclusion map $i$ as follows:
\begin{align*}
  \pi&:\left \langle x_{1},\cdots,x_{\frac{n-2}{2}},y_{k+1},\cdots,y_{\frac{n-2}{2}} \right \rangle \longrightarrow \left \langle x_{k+1},y_{k+1},\cdots,x_{\frac{n-2}{2}},y_{\frac{n-2}{2}} \right \rangle\\
  i&:g(S) \longrightarrow \left \langle x_{1},\cdots,x_{\frac{n-2}{2}},y_{k+1},\cdots,y_{\frac{n-2}{2}} \right \rangle. 
\end{align*}
Since $g(S)\cap g(W)=0$, it follows that ker$(\pi)=g(W)$. This induces an injective embedding $\pi \circ i$ by the following diagram:
\begin{center}
\begin{tikzcd}
g(S) \arrow[rd,hook,"\pi \circ i"'] \arrow[r,hook, "i"] & \arrow[d,"\pi"] \left \langle x_{1},\cdots,x_{\frac{n-2}{2}},y_{k+1},\cdots,y_{\frac{n-2}{2}} \right \rangle \\
& \left \langle x_{k+1},y_{k+1},\cdots,x_{\frac{n-2}{2}},y_{\frac{n-2}{2}} \right \rangle
\end{tikzcd}
\end{center}
Since $\text{diag}(1^{n-2k-2})=\left \langle x_{k+1},y_{k+1},\cdots,x_{\frac{n-2}{2}},y_{\frac{n-2}{2}} \right \rangle$, it follows that $S$ is embedded in diag$(1^{n-2k-2})\oplus \text{diag}(1,\lambda)=\text{diag}(1^{n-2k-1},\lambda)$.

\smallskip
Conversely, suppose that $S$ is embedded in $ \text{diag}(1^{n-2k-1},\lambda)$. Note that \[\text{diag}(1^{n-1},\lambda)=\text{diag}(1^{n-2k-1},\lambda)\oplus \bigoplus_{i=1}^{k}\mathbb{H} \supseteq \text{diag}(1^{n-2k-1},\lambda)\oplus \left \langle x_{1},\cdots,x_{k} \right \rangle,\] 
where $\mathbb{H}=\left \langle x_{i},y_{i} \right \rangle$, $B(x_{i},y_{i})=\delta_{i,j}$ and $B(x_{i},x_{j})=B(y_{i},y_{j})=0$. Since $W$ is isometrically isomorphic to $\left \langle x_{1},\cdots,x_{k} \right \rangle$, this completes the proof. 

\smallskip
When $n$ is odd, we omit the proof because it is similar to the case of $n$ being even.
\end{proof}

Suppose that a quadratic subspace $W$ of $(\mathbb{Q}_{p}^{n},\text{diag}(1^{n}))$ is isometrically isomorphic to diag$(0^{k})$. Then dim$W\leq n/2$. If $n$ is even, unlike the case in which the embedded space is $(\mathbb{Q}_{p}^{n},\text{diag}(1^{n}))$, dim$W$ cannot be $n/2$. If dim$W=n/2$, by Proposition \ref{hyper}, there is a totally isotropic subspace $W$ such that $\left \langle W,W' \right \rangle=\bigoplus_{i=1}^{n/2}\mathbb{H}=\text{diag}(1^{n-1},\lambda)$, which contradicts the fact that the discriminant of $\bigoplus_{i=1}^{n/2}\mathbb{H}$ is a square but the discriminant of $\text{diag}(1^{n-1},\lambda)$ is a nonsquare. 

\begin{corollary}
Suppose that $p \equiv 1$ mod $4$. For each $n$ and $k$, 
\begin{itemize}
    \item diag$(1)\oplus$diag$(0^{k})$ can be embedded in diag$(1^{n-1},\lambda)$ if and only if $k \leq \frac{n-2}{2}$,
    \item diag$(\lambda)\oplus$diag$(0^{k})$ can be embedded in diag$(1^{n-1},\lambda)$ if and only if $k \leq \frac{n-1}{2}$,
    \item diag$(p)\oplus$diag$(0^{k})$ can be embedded in diag$(1^{n-1},\lambda)$ if and only if $k \leq \frac{n-3}{2}$,
    \item diag$(\lambda p)\oplus$diag$(0^{k})$ can be embedded in diag$(1^{n-1},\lambda)$ if and only if $k \leq \frac{n-3}{2}$,
    \item diag$(1,1)\oplus$diag$(0^{k})$ can be embedded in diag$(1^{n-1},\lambda)$ if and only if $k \leq \frac{n-3}{2}$,
    \item diag$(1,\lambda )\oplus$diag$(0^{k})$ can be embedded in diag$(1^{n-1},\lambda)$ if and only if $k \leq \frac{n-2}{2}$,
    \item diag$(1,p)\oplus$diag$(0^{k})$ can be embedded in diag$(1^{n-1},\lambda)$ if and only if $k \leq \frac{n-4}{2}$,
    \item diag$(1,\lambda p)\oplus$diag$(0^{k})$ can be embedded in diag$(1^{n-1},\lambda)$ if and only if $k \leq \frac{n-4}{2}$,
    \item diag$(\lambda,p)\oplus$diag$(0^{k})$ can be embedded in diag$(1^{n-1},\lambda)$ if and only if $k \leq \frac{n-3}{2}$,
    \item diag$(\lambda,\lambda p)\oplus$diag$(0^{k})$ can be embedded in diag$(1^{n-1},\lambda)$ if and only if $k \leq \frac{n-3}{2}$,
    \item diag$(p,\lambda p)\oplus$diag$(0^{k})$ can be embedded in diag$(1^{n-1},\lambda)$ if and only if $k \leq \frac{n-5}{2}$,
    \item diag$(1,1,1)\oplus$diag$(0^{k})$ can be embedded in diag$(1^{n-1},\lambda)$ if and only if $k \leq \frac{n-4}{2}$, 
    \item diag$(1,1,\lambda )\oplus$diag$(0^{k})$ can be embedded in diag$(1^{n-1},\lambda)$ if and only if $k \leq \frac{n-3}{2}$,
    \item diag$(1,\lambda ,p)\oplus$diag$(0^{k})$ can be embedded in diag$(1^{n-1},\lambda)$ if and only if $k \leq \frac{n-4}{2}$,
    \item diag$(1,\lambda ,\lambda p)\oplus$diag$(0^{k})$ can be embedded in diag$(1^{n-1},\lambda)$ if and only if $k \leq \frac{n-4}{2}$,
    \item diag$(1,1,p)\oplus$diag$(0^{k})$ can be embedded in diag$(1^{n-1},\lambda)$ if and only if $k \leq \frac{n-5}{2}$,
    \item diag$(1,1,\lambda p)\oplus$diag$(0^{k})$ can be embedded in diag$(1^{n-1},\lambda)$ if and only if $k \leq \frac{n-5}{2}$,
    \item diag$(\lambda ,p,\lambda p)\oplus$diag$(0^{k})$ can be embedded in diag$(1^{n-1},\lambda)$ if and only if $k \leq \frac{n-5}{2}$,
    \item diag$(1,p,\lambda p)\oplus$diag$(0^{k})$ can be embedded in diag$(1^{n-1},\lambda)$ if and only if $k \leq \frac{n-6}{2}$.
\end{itemize}
\end{corollary}

Since any nondegenerate quadratic space of dimension $k$ is isometrically isomorphic to diag$(1^{k-3})\oplus Q$, where $Q$ is a $3$-dimensional nondegenerate quadratic space, given any quadratic space $V$, we can determine if $V$ can be embedded in diag$(1^{n-1},\lambda)$ when $p \equiv 1$ mod $4$.

\subsection{Embeddings in $(\mathbb{Q}_{p}^{n},\text{diag}(1^{n} ))$ when $p\equiv 3$ mod $4$}
Let us consider $1$-dimensional nondegenerate quadratic spaces. Suppose that $n=2$. Then there are $2$ types of nondegenerate quadratic subspaces of $(\mathbb{Q}_{p}^{2},\text{diag}(1,1))$: 
\[ \text{diag}(1)=\left \langle (1,0) \right \rangle,
    ~\text{diag}(\lambda)=\left \langle (a,b) \right \rangle .\]
\begin{lemma}
If $p \equiv 3$ mod $4$, diag$(p)$ does not exist in $(\mathbb{Q}_{p}^{2},\text{diag}(1,1))$.
\end{lemma}

\begin{proof}
Suppose that diag$(p)$ exists. By comparing the discriminant of diag$(1,1)$ and diag$(p)\oplus \text{diag}(p)$, we obtain $\text{diag}(1,1)=\text{diag}(p)\oplus \text{diag}(p)$, which contradicts the fact that the Hasse-invariant of diag$(1,1)$ is $1$, but diag$(p,p)$ is $-1$.
\end{proof}
Similarly, it is not hard to show that diag$(\lambda p)$ does not exist in $(\mathbb{Q}_{p}^{2},\text{diag}(1,1))$. Furthermore, by inserting $0$ $(n-2)$ times to each vector, the spaces diag$(1)$ and diag$(\lambda)$ exist when $n \geq 3$. If $n=3$, the other types, diag$(p)$ and diag$(\lambda p)$, also exist as follows:
\[
    \text{diag}(p)=\left \langle (a,b,c) \right \rangle, ~\text{diag}(\lambda p)=\left \langle (a,b,d) \right \rangle.
\]
Thus all $4$ types of $1$-dimensional nondegenerate quadratic spaces can be embedded in $(\mathbb{Q}_{p}^{n},\text{diag}(1^{n}))$ when $n \geq 4$. 

\smallskip
Let us think about $2$-dimensional nondegenerate quadratic spaces over $\mathbb{Q}_{p}$. Suppose that $n=3$. Since there are $4$ types of $1$-dimensional nondegenerate quadratic subspaces of $(\mathbb{Q}_{p}^{3},\text{diag}(1^{3}))$, there should be $4$ types of $2$-dimensional nondegenerate quadratic subspaces of $(\mathbb{Q}_{p}^{3},\text{diag}(1^{3}))$. We find the following $4$ types in Theorem \ref{cla1}:
\begin{itemize}
    \item diag$(1,1)=\left \langle (1,0,0),(0,1,0) \right \rangle$,
    \item diag$(1,\lambda)=\left \langle (1,0,0),(0,a,b) \right \rangle$,
    \item diag$(\lambda,p)=\left \langle (a,b,0),(-b,a,c) \right \rangle$,
    \item diag$(\lambda,\lambda p)=\left \langle (a,b,0),(-b,a,d) \right \rangle$.
\end{itemize}
If $n=4$, then $3$ of the other cases also exist:
\begin{itemize}
    \item diag$(p,p)=\left \langle (a,b,c,0),(-b,a,0,c) \right \rangle$,
    \item diag$(1,p)=\left \langle (1,0,0,0),(0,a,b,c) \right \rangle$,
    \item diag$(1,\lambda p)=\left \langle (1,0,0,0),(0,a,b,d) \right \rangle$.
\end{itemize}
If $n \geq 5$, all $7$ types of nondegenerate $2$-dimensional quadratic spaces can be embedded in $(\mathbb{Q}_{p}^{n},\text{diag}(1^{n}))$.

\smallskip
Let us consider $3$-dimensional nondegenerate quadratic spaces. Suppose that $n=4$. Since there are $4$ types of $1$-dimensional nondegenerate quadratic subspaces of $(\mathbb{Q}_{p}^{4},\text{diag}(1^{4}))$, there have to be $4$ types of $3$-dimensional nondegenerate quadratic subspaces of $(\mathbb{Q}_{p}^{4},\text{diag}(1^{4}))$. For diag$(1,1,1)$, we can construct diag$(1,1,1)=\left \langle (1,0,0,0),(0,1,0,0),(0,0,1,0) \right \rangle$. We verify the existence of other types by comparing the discriminants and the Hasse-invariants. The quadratic space diag$(p,p,p)$ exists in $(\mathbb{Q}_{p}^{4},\text{diag}(1^{4}))$ because its perpendicular space $(\text{diag}(p))^{\perp}$ is diag$(p,p,p)$. Furthermore, diag$(1,\lambda,p )$ and diag$(1,p,\lambda p)$ also exist by verifying $(\text{diag}(\lambda ))^{\perp}$ and $(\text{diag}(\lambda p))^{\perp}$. If $n=5$, since there are $7$ types of $2$-dimensional nondegenerate quadratic subspaces of $(\mathbb{Q}_{p}^{5},\text{diag}(1^{5}))$, there must be $7$ types of $3$-dimensional nongenerate quadratic subspace in $(\mathbb{Q}_{p}^{5},\text{diag}(1^{5}))$. Similarly, it is not hard to check that there are the following $7$ types containing diag$(p,p,p)$:
\begin{itemize}
    \item diag$(1,1,1)=\left \langle (1,0,0,0,0),(0,1,0,0,0),(0,0,1,0,0) \right \rangle$,
    \item diag$(1,p,p)=\left \langle (1,0,0,0,0),(0,a,b,c,0),(0,-b,a,0,c) \right \rangle$,
    \item diag$(1,1,p)=\left \langle (1,0,0,0,0),(0,1,0,0,0),(0,0,a,b,c) \right \rangle$,
    \item diag$(1,\lambda,p)=\left \langle (1,0,0,0,0),(0,a,b,d,0),(0,-b,a,0,c) \right \rangle$.
    \item diag$(1,1,\lambda p)=\left \langle (1,0,0,0,0),(0,1,0,0,0),(0,0,a,b,d) \right \rangle$.
    \item diag$(1,p ,\lambda p)=\left \langle (1,0,0,0,0),(0,a,b,c,0),(0,-b,a,0,d) \right \rangle$.
\end{itemize}
For diag$(p,p,p)$, this can be embedded in $(\mathbb{Q}_{p}^{n},\text{diag}(1^{n}))$ when $n=4$, it follows that it does so when $n=5$. If $n=6$, diag$(p,p,\lambda)=\left \langle (a,b,c,0,0,0),(-b,a,0,c,0,0),(0,0,0,0,a,b) \right \rangle$ also can exist in $(\mathbb{Q}_{p}^{6},\text{diag}(1^{6}))$. If $n \geq 6$, then all $8$ types in Theorem \ref{cla2} can be embedded in $(\mathbb{Q}_{p}^{n},\text{diag}(1^{n}))$.

\smallskip
We give the conditions for degenerate quadratic spaces. Since the proofs are similar to Theorem \ref{main} and Theorem \ref{main2}, we omit the proofs.
\begin{theorem}
Suppose that $p\equiv 3$ mod $4$. Let $A$ be a degenerate quadratic space and let us denote $A$ by $\text{diag}(0^{k}) \oplus S$, where $S$ is a nondegenerate quadratic space. 
\begin{itemize}
    \item If $n$ is odd and $(n-2k-1)/2$ is even or $n$ is even and $(n-2k)/2$ is even, then $A$ is embedded in diag$(1^{n})$ if and only if $S$ is embedded in $\text{diag}(1^{n-2k})$ (and $k \leq (n-2)/2$ when $n=4l+2$).
    \item If $n$ is odd and $(n-2k-1)/2$ is odd or $n$ is even and $(n-2k)/2$ is odd, then $A$ is embedded in diag$(1^{n})$ if and only if $S$ is embedded in $\text{diag}(1^{n-2k-1},\lambda)$ (and $k \leq (n-2)/2$ when $n=4l+2$).
\end{itemize}
\end{theorem}

Suppose that a quadratic subspace $W$ of $(\mathbb{Q}_{p}^{n},\text{diag}(1^{n}))$ is isometrically isomorphic to diag$(0^{k})$. Then dim$W\leq n/2$. If $n$ is even, $n=4l+2$ and is embedded in $(\mathbb{Q}_{p}^{n},\text{diag}(1^{n}))$, dim$W$ cannot be $n/2$. If $n=4l$ and dim$W=n/2$, then $W$ can be constructed by
\[\left \langle (a,b,e,0,\cdots,0),(-b,a,0,e,0,\cdots,0),\cdots,(0,\cdots,0,a,b,e,0),(0,\cdots,0,-b,a,0,e) \right \rangle.\]

\begin{corollary}
Suppose that $p \equiv 3$ mod $4$. If $n$ is odd and $(n-2k-1)/2$ is even, or if $n$ is even and $(n-2k)/2$ is even, then we have
\begin{itemize}
    \item diag$(1)\oplus$diag$(0^{k})$ can be embedded in diag$(1^{n})$ if and only if $k \leq \frac{n-1}{2}$,
    \item diag$(\lambda)\oplus$diag$(0^{k})$ can be embedded in diag$(1^{n})$ if and only if $k \leq \frac{n-2}{2}$,
    \item diag$(p)\oplus$diag$(0^{k})$ can be embedded in diag$(1^{n})$ if and only if $k \leq \frac{n-3}{2}$,
    \item diag$(\lambda p)\oplus$diag$(0^{k})$ can be embedded in diag$(1^{n})$ if and only if $k \leq \frac{n-3}{2}$,
    \item diag$(1,1)\oplus$diag$(0^{k})$ can be embedded in diag$(1^{n})$ if and only if $k \leq \frac{n-2}{2}$,
    \item diag$(1,\lambda )\oplus$diag$(0^{k})$ can be embedded in diag$(1^{n})$ if and only if $k \leq \frac{n-3}{2}$,
    \item diag$(\lambda,p)\oplus$diag$(0^{k})$ can be embedded in diag$(1^{n})$ if and only if $k \leq \frac{n-3}{2}$,
    \item diag$(\lambda ,\lambda p)\oplus$diag$(0^{k})$ can be embedded in diag$(1^{n})$ if and only if $k \leq \frac{n-3}{2}$,
    \item diag$(p,p)\oplus$diag$(0^{k})$ can be embedded in diag$(1^{n})$ if and only if $k \leq \frac{n-4}{2}$,
    \item diag$(1, p)\oplus$diag$(0^{k})$ can be embedded in diag$(1^{n})$ if and only if $k \leq \frac{n-4}{2}$,
    \item diag$(1, \lambda  p)\oplus$diag$(0^{k})$ can be embedded in diag$(1^{n})$ if and only if $k \leq \frac{n-4}{2}$,
    \item diag$(1,1,1)\oplus$diag$(0^{k})$ can be embedded in diag$(1^{n})$ if and only if $k \leq \frac{n-3}{2}$, 
    \item diag$(p,p,p)\oplus$diag$(0^{k})$ can be embedded in diag$(1^{n})$ if and only if $k \leq \frac{n-4}{2}$,
    \item diag$(1,\lambda, p )\oplus$diag$(0^{k})$ can be embedded in diag$(1^{n})$ if and only if $k \leq \frac{n-4}{2}$,
    \item diag$(1,p,\lambda p)\oplus$diag$(0^{k})$ can be embedded in diag$(1^{n})$ if and only if $k \leq \frac{n-4}{2}$,
    \item diag$(1,p, p)\oplus$diag$(0^{k})$ can be embedded in diag$(1^{n})$ if and only if $k \leq \frac{n-5}{2}$,
    \item diag$(1,1 ,p)\oplus$diag$(0^{k})$ can be embedded in diag$(1^{n})$ if and only if $k \leq \frac{n-5}{2}$,
    \item diag$(1,1 ,\lambda p)\oplus$diag$(0^{k})$ can be embedded in diag$(1^{n})$ if and only if $k \leq \frac{n-5}{2}$,
    \item diag$(p,p,\lambda )\oplus$diag$(0^{k})$ can be embedded in diag$(1^{n})$ if and only if $k \leq \frac{n-6}{2}$.
\end{itemize}

\smallskip
If $n$ is odd and $(n-2k-1)/2$ is odd, or if $n$ is even and $(n-k)/2$ is odd, then we have
\begin{itemize}
    \item diag$(1)\oplus$diag$(0^{k})$ can be embedded in diag$(1^{n})$ if and only if $k \leq \frac{n-2}{2}$,
    \item diag$(\lambda)\oplus$diag$(0^{k})$ can be embedded in diag$(1^{n})$ if and only if $k \leq \frac{n-1}{2}$,
    \item diag$(p)\oplus$diag$(0^{k})$ can be embedded in diag$(1^{n})$ if and only if $k \leq \frac{n-2}{2}$,
    \item diag$(\lambda p)\oplus$diag$(0^{k})$ can be embedded in diag$(1^{n})$ if and only if $k \leq \frac{n-2}{2}$,
    \item diag$(1,1)\oplus$diag$(0^{k})$ can be embedded in diag$(1^{n})$ if and only if $k \leq \frac{n-3}{2}$,
    \item diag$(1,\lambda )\oplus$diag$(0^{k})$ can be embedded in diag$(1^{n})$ if and only if $k \leq \frac{n-2}{2}$,
    \item diag$(1,p)\oplus$diag$(0^{k})$ can be embedded in diag$(1^{n})$ if and only if $k \leq \frac{n-3}{2}$,
    \item diag$(1 ,\lambda p)\oplus$diag$(0^{k})$ can be embedded in diag$(1^{n})$ if and only if $k \leq \frac{n-3}{2}$,
    \item diag$(\lambda,p)\oplus$diag$(0^{k})$ can be embedded in diag$(1^{n})$ if and only if $k \leq \frac{n-4}{2}$,
    \item diag$(\lambda, \lambda p)\oplus$diag$(0^{k})$ can be embedded in diag$(1^{n})$ if and only if $k \leq \frac{n-4}{2}$,
    \item diag$(p,   p)\oplus$diag$(0^{k})$ can be embedded in diag$(1^{n})$ if and only if $k \leq \frac{n-5}{2}$,
    \item diag$(1,1,1)\oplus$diag$(0^{k})$ can be embedded in diag$(1^{n})$ if and only if $k \leq \frac{n-4}{2}$, 
    \item diag$(1,p,\lambda p)\oplus$diag$(0^{k})$ can be embedded in diag$(1^{n})$ if and only if $k \leq \frac{n-3}{2}$,
    \item diag$(1,1,\lambda p )\oplus$diag$(0^{k})$ can be embedded in diag$(1^{n})$ if and only if $k \leq \frac{n-4}{2}$,
    \item diag$(1,1,p)\oplus$diag$(0^{k})$ can be embedded in diag$(1^{n})$ if and only if $k \leq \frac{n-4}{2}$,
    \item diag$(p,p, p)\oplus$diag$(0^{k})$ can be embedded in diag$(1^{n})$ if and only if $k \leq \frac{n-5}{2}$,
    \item diag$(1,\lambda ,p)\oplus$diag$(0^{k})$ can be embedded in diag$(1^{n})$ if and only if $k \leq \frac{n-5}{2}$,
    \item diag$(p,p ,\lambda )\oplus$diag$(0^{k})$ can be embedded in diag$(1^{n})$ if and only if $k \leq \frac{n-5}{2}$,
    \item diag$(1,p,p )\oplus$diag$(0^{k})$ can be embedded in diag$(1^{n})$ if and only if $k \leq \frac{n-6}{2}$.
\end{itemize}
\end{corollary}

Since any nondegenerate quadratic space of dimension $k$ is isometrically isomorphic to diag$(1^{k-3})\oplus Q$, where $Q$ is a $3$-dimensional nondegenerate quadratic space, given any quadratic space $V$, we can determine if $V$ can be embedded in diag$(1^{n})$ when $p \equiv 3$ mod $4$.

\subsection{Embeddings in $(\mathbb{Q}_{p}^{n},\text{diag}(1^{n-1},\lambda ))$ when $p\equiv 3$ mod $4$}
Let us consider $1$-dimensional nondegenerate quadratic spaces. Suppose that $n=2$. Then there are $4$ types of quadratic nondegenerate subspaces of $(\mathbb{Q}_{p}^{2},\lambda \text{dot}_{2})$. First, we have the following two nondegenerate subspaces, diag$(1)=\left \langle (1,0) \right \rangle$, and  diag$(\lambda)=(\text{diag}(1))^{\perp}$. Unlike the case in which the embedded space is Euclidean and $p\equiv 3$ mod $4$, diag$(p)$ exists in $(\mathbb{Q}_{p}^{2},\text{diag}(1,\lambda))$. 

\begin{lemma}
If $p \equiv 3$ mod $4$, then diag$(p)$ exist in $(\mathbb{Q}_{p}^{2},\text{diag}(1,\lambda))$.
\end{lemma}

\begin{proof}
We only need to show that there are $z,x$ in $\mathbb{Q}_{p}$ such that $z^{2}+\lambda w^{2}=p$. We may assume $\lambda$ is $-1$ when $p \equiv 3$ mod $4$. Then we have $(z-w)(z+w)=p$, and by changing of coordinates from $s=z-w$ and $t=z+w$, this completes the existence of diag$(p)$.
\end{proof}
Furthemore, since diag$(\lambda p)=(\text{diag}(p))^{\perp}$, thus diag$(\lambda p)$ also exists in $(\mathbb{Q}_{p}^{2},\text{diag}(1,\lambda))$. Therefore, all $4$ types exist in $(\mathbb{Q}_{p}^{n},\text{diag}(1^{n-1},\lambda))$ when $n \geq 3$. 

\smallskip
Let us think about $2$-dimensional quadratic spaces over $\mathbb{Q}_{p}$. Suppose that $n=3$. Since there are $4$ types of $1$-dimensional nondegenerate quadratic subspaces of $(\mathbb{Q}_{p}^{3},\text{diag}(1,1,\lambda))$, there should be $4$ types of $2$-dimensional quadratic subspaces of $(\mathbb{Q}_{p}^{3},\text{diag}(1,1,\lambda))$. We find the following $4$ types in Theorem \ref{cla1}:
\[\text{diag}(1,1)=(\text{diag}(\lambda))^{\perp},\text{diag}(1,\lambda)=(\text{diag}(1))^{\perp}, \text{diag}(1,p)=(\text{diag}(\lambda p))^{\perp},\text{diag}(1,\lambda p)=(\text{diag}(p))^{\perp}.\]
If $n=4$, then the following $2$ cases also exist: diag$(\lambda ,p)=(\text{diag}(1,p))^{\perp}$ and diag$(\lambda ,\lambda p)=(\text{diag}(1,\lambda p))^{\perp}$. However, diag$(p,p)$ does not exist. By comparing the discriminants, the only possible orthogonal complement of diag$(p,p)$ is diag$(1,\lambda)$, but the Hasse-invariant of diag$(p,p)\oplus \text{diag}(1,\lambda)$ is $-1$. This contradicts the fact that the Hasse-invariant of diag$(1^{3},\lambda)$ is $1$. Suppose that $n=5$. Then we have diag$(p,p)=\left \langle (a,b,c,0,0),(-b,a,0,c,0) \right \rangle$ and thus all $7$ types exist in $(\mathbb{Q}_{p}^{5},\text{diag}(1^{4},\lambda))$. Similarly, if $n \geq 5$, all $7$ types of nondegenerate $2$-dimensional quadratic spaces can be embedded in $(\mathbb{Q}_{p}^{n},\text{diag}(1^{n-1},\lambda))$.

\smallskip
Let us consider $3$-dimensional nondegenerate quadratic spaces. Suppose that $n=4$. Since there are $4$ types of $1$-dimensional nondegenerate quadratic subspaces of $(\mathbb{Q}_{p}^{4},\text{diag}(1^{3},\lambda))$, there have to be $4$ types of $3$-dimensional quadratic subspaces of $(\mathbb{Q}_{p}^{4},\text{diag}(1^{3},\lambda))$. For diag$(1,1,1)$, we can construct diag$(1,1,1)=\left \langle (1,0,0,0),(0,1,0,0),(0,0,1,0) \right \rangle$. But for other cases, it is hard to find vectors satisfying the conditions. By comparing the discriminants and the Hasse-invariants, we find 
\[\text{diag}(1,p,\lambda p)=(\text{diag}(1))^{\perp},\text{diag}(1,1,\lambda p)=(\text{diag}(p))^{\perp},\text{ and } \text{diag}(1,1 ,p)=(\text{diag}(\lambda p))^{\perp}\]
exist when $n=4$. If $n=5$, since there are $7$ types of $2$-dimensional nondegenerate quadratic spaces $(\mathbb{Q}_{p}^{5},\text{diag}(1^{4},\lambda))$, there must be $7$ types of $3$-dimensional quadratic subspace in $(\mathbb{Q}_{p}^{5},\text{diag}(1^{4},\lambda))$. It is not hard to check that the following $7$ types exist except for diag$(1,p,p)$:
\begin{itemize}
    \item diag$(1,1,1)=(\text{diag}(1,\lambda ))^{\perp}$,
    \item diag$(1,1,p)=(\text{diag}(\lambda,p))^{\perp}$,
    \item diag$(p,p,p)=(\text{diag}(1,\lambda p))^{\perp}$,
    \item diag$(1,\lambda,p)=(\text{diag}(1,p))^{\perp}$,
    \item diag$(1,1,\lambda p)=(\text{diag}(\lambda ,\lambda p))^{\perp}$,
    \item diag$(p,p ,\lambda )=(\text{diag}(1,1))^{\perp}$,
    \item diag$(1,p ,\lambda p)=(\text{diag}(p,p))^{\perp}$.
\end{itemize}
If $n=6$, we also have diag$(1,p,p)=(\text{diag}(1,p,\lambda p))^{\perp}$. Therefore, all $8$ cases can be embedded in $(\mathbb{Q}_{p}^{n},\text{diag}(1^{n-1},\lambda))$ when $n \geq 6$.

\smallskip
We now give the conditions for degenerate quadratic spaces. Since the proofs are similar to Theorem \ref{main} and Theorem \ref{main2}, we omit the proofs.

\begin{theorem}
Suppose that $p\equiv 3$ mod $4$. Let $A$ be a degenerate quadratic space and let us denote $A$ by $\text{diag}(0^{k}) \oplus S$, where $S$ is a nondegenerate quadratic space. 
\begin{itemize}
    \item If $n$ is odd and $(n-2k-1)/2$ is even, or $n$ is even and $(n-2k)/2$ is even, then $A$ is embedded in diag$(1^{n-1},\lambda)$ if and only if $S$ is embedded in $\text{diag}(1^{n-2k-1},\lambda )$ (and $k\leq (n-2)/2$ when $n=4l$).
    \item If $n$ is odd and $(n-2k-1)/2$ is odd, or $n$ is even and $(n-2k)/2$ is odd, then $A$ is embedded in diag$(1^{n-1},\lambda)$ if and only if $S$ is embedded in $\text{diag}(1^{n-2k})$ (and $k\leq (n-2)/2$ when $n=4l$).
\end{itemize}
\end{theorem}
Suppose that a quadratic subspace $W$ of $(\mathbb{Q}_{p}^{n},\text{diag}(1^{n-1},\lambda))$ is isometrically isomorphic to diag$(0^{k})$. Then dim$W\leq n/2$. If $n=4l$, the dim$W$ cannot be $n/2$. If $n=4l+2$, dim$W=n/2$ is possible. Note that $\mathbb{H}=\text{diag}(1,\lambda)$. Thus we have diag$(1^{n-1},\lambda)=\text{diag}(1,\lambda)\oplus \text{diag}(1^{4l})=\bigoplus_{i=1}^{2l+1}\mathbb{H}$. Let $\mathbb{H}=\left \langle x_{i},y_{i} \right \rangle$. Then $W$ is isometrically isomorphic to $\left \langle x_{1},\cdots,x_{2l+1} \right \rangle$, which has the dimension $n/2$.

\begin{corollary}
Suppose that $p \equiv 3$ mod $4$. If $n$ is odd and $(n-2k-1)/2$ is even, or if $n$ is even and $(n-2k)/2$ is even, then we have
\begin{itemize}
    \item diag$(1)\oplus$diag$(0^{k})$ can be embedded in diag$(1^{n-1},\lambda)$ if and only if $k \leq \frac{n-2}{2}$,
    \item diag$(\lambda)\oplus$diag$(0^{k})$ can be embedded in diag$(1^{n-1},\lambda)$ if and only if $k \leq \frac{n-1}{2}$,
    \item diag$(p)\oplus$diag$(0^{k})$ can be embedded in diag$(1^{n-1},\lambda)$ if and only if $k \leq \frac{n-2}{2}$,
    \item diag$(\lambda p)\oplus$diag$(0^{k})$ can be embedded in diag$(1^{n-1},\lambda)$ if and only if $k \leq \frac{n-2}{2}$,
    \item diag$(1,1)\oplus$diag$(0^{k})$ can be embedded in diag$(1^{n-1},\lambda)$ if and only if $k \leq \frac{n-3}{2}$,
    \item diag$(1,\lambda )\oplus$diag$(0^{k})$ can be embedded in diag$(1^{n-1},\lambda)$ if and only if $k \leq \frac{n-2}{2}$,
    \item diag$(1,p)\oplus$diag$(0^{k})$ can be embedded in diag$(1^{n-1},\lambda)$ if and only if $k \leq \frac{n-3}{2}$,
    \item diag$(1 ,\lambda p)\oplus$diag$(0^{k})$ can be embedded in diag$(1^{n-1},\lambda)$ if and only if $k \leq \frac{n-3}{2}$,
    \item diag$(\lambda,p)\oplus$diag$(0^{k})$ can be embedded in diag$(1^{n-1},\lambda)$ if and only if $k \leq \frac{n-4}{2}$,
    \item diag$(\lambda, \lambda p)\oplus$diag$(0^{k})$ can be embedded in diag$(1^{n-1},\lambda)$ if and only if $k \leq \frac{n-4}{2}$,
    \item diag$(1, \lambda  p)\oplus$diag$(0^{k})$ can be embedded in diag$(1^{n-1},\lambda)$ if and only if $k \leq \frac{n-5}{2}$,
    \item diag$(1,1,1)\oplus$diag$(0^{k})$ can be embedded in diag$(1^{n-1},\lambda)$ if and only if $k \leq \frac{n-4}{2}$, 
    \item diag$(1,p,\lambda p)\oplus$diag$(0^{k})$ can be embedded in diag$(1^{n-1},\lambda)$ if and only if $k \leq \frac{n-3}{2}$,
    \item diag$(1,1,\lambda p )\oplus$diag$(0^{k})$ can be embedded in diag$(1^{n-1},\lambda)$ if and only if $k \leq \frac{n-4}{2}$,
    \item diag$(1,1,p)\oplus$diag$(0^{k})$ can be embedded in diag$(1^{n-1},\lambda)$ if and only if $k \leq \frac{n-4}{2}$,
    \item diag$(p,p, p)\oplus$diag$(0^{k})$ can be embedded in diag$(1^{n-1},\lambda)$ if and only if $k \leq \frac{n-5}{2}$,
    \item diag$(1,\lambda ,p)\oplus$diag$(0^{k})$ can be embedded in diag$(1^{n-1},\lambda)$ if and only if $k \leq \frac{n-5}{2}$,
    \item diag$(p,p ,\lambda )\oplus$diag$(0^{k})$ can be embedded in diag$(1^{n-1},\lambda)$ if and only if $k \leq \frac{n-5}{2}$,
    \item diag$(1,p,p )\oplus$diag$(0^{k})$ can be embedded in diag$(1^{n-1},\lambda)$ if and only if $k \leq \frac{n-6}{2}$.
\end{itemize}

\smallskip
If $n$ is odd and $(n-2k-1)/2$ is odd, or if $n$ is even and $(n-k)/2$ is odd, then we have

\begin{itemize}
    \item diag$(1)\oplus$diag$(0^{k})$ can be embedded in diag$(1^{n-1},\lambda)$ if and only if $k \leq \frac{n-1}{2}$,
    \item diag$(\lambda)\oplus$diag$(0^{k})$ can be embedded in diag$(1^{n-1},\lambda)$ if and only if $k \leq \frac{n-2}{2}$,
    \item diag$(p)\oplus$diag$(0^{k})$ can be embedded in diag$(1^{n-1},\lambda)$ if and only if $k \leq \frac{n-3}{2}$,
    \item diag$(\lambda p)\oplus$diag$(0^{k})$ can be embedded in diag$(1^{n-1},\lambda)$ if and only if $k \leq \frac{n-3}{2}$,
    \item diag$(1,1)\oplus$diag$(0^{k})$ can be embedded in diag$(1^{n-1},\lambda)$ if and only if $k \leq \frac{n-2}{2}$,
    \item diag$(1,\lambda )\oplus$diag$(0^{k})$ can be embedded in diag$(1^{n-1},\lambda)$ if and only if $k \leq \frac{n-3}{2}$,
    \item diag$(\lambda,p)\oplus$diag$(0^{k})$ can be embedded in diag$(1^{n-1},\lambda)$ if and only if $k \leq \frac{n-3}{2}$,
    \item diag$(\lambda ,\lambda p)\oplus$diag$(0^{k})$ can be embedded in diag$(1^{n-1},\lambda)$ if and only if $k \leq \frac{n-3}{2}$,
    \item diag$(p,p)\oplus$diag$(0^{k})$ can be embedded in diag$(1^{n-1},\lambda)$ if and only if $k \leq \frac{n-4}{2}$,
    \item diag$(1, p)\oplus$diag$(0^{k})$ can be embedded in diag$(1^{n-1},\lambda)$ if and only if $k \leq \frac{n-4}{2}$,
    \item diag$(1, \lambda p p)\oplus$diag$(0^{k})$ can be embedded in diag$(1^{n-1},\lambda)$ if and only if $k \leq \frac{n-4}{2}$,
    \item diag$(1,1,1)\oplus$diag$(0^{k})$ can be embedded in diag$(1^{n-1},\lambda)$ if and only if $k \leq \frac{n-3}{2}$, 
    \item diag$(p,p,p)\oplus$diag$(0^{k})$ can be embedded in diag$(1^{n-1},\lambda)$ if and only if $k \leq \frac{n-4}{2}$,
    \item diag$(1,\lambda, p )\oplus$diag$(0^{k})$ can be embedded in diag$(1^{n-1},\lambda)$ if and only if $k \leq \frac{n-4}{2}$,
    \item diag$(1,p,\lambda p)\oplus$diag$(0^{k})$ can be embedded in diag$(1^{n-1},\lambda)$ if and only if $k \leq \frac{n-4}{2}$,
    \item diag$(1,p, p)\oplus$diag$(0^{k})$ can be embedded in diag$(1^{n-1},\lambda)$ if and only if $k \leq \frac{n-5}{2}$,
    \item diag$(1,1 ,p)\oplus$diag$(0^{k})$ can be embedded in diag$(1^{n-1},\lambda)$ if and only if $k \leq \frac{n-5}{2}$,
    \item diag$(1,1 ,\lambda p)\oplus$diag$(0^{k})$ can be embedded in diag$(1^{n-1},\lambda)$ if and only if $k \leq \frac{n-5}{2}$,
    \item diag$(p,p,\lambda )\oplus$diag$(0^{k})$ can be embedded in diag$(1^{n-1},\lambda)$ if and only if $k \leq \frac{n-6}{2}$.
\end{itemize}
\end{corollary}
Notice that any nondegenerate quadratic space of dimension $k$ is isometrically isomorphic to diag$(1^{k-3})\oplus Q$, where $Q$ is a $3$-dimensional nondegenerate quadratic space. At this point, we have all embedding conditions for nondegenerate and degenerate quadratic spaces. Therefore, given any quadratic space $V$, we can determine if $V$ can be embedded in diag$(1^{n-1},\lambda)$ when $p \equiv 3$ mod $4$.

\end{document}